\documentclass[a4paper,11pt]{article}
\usepackage[TS1,OT1]{fontenc}
\usepackage{amsmath,mathtools}
\usepackage{amsfonts}
\usepackage{amssymb}
\usepackage{mathrsfs}
\usepackage{hyperref}
\usepackage[retainorgcmds]{IEEEtrantools}
\usepackage{enumerate}
\usepackage{enumitem}
\usepackage{xfrac}
\usepackage{amsthm}
\usepackage{fullpage}
\usepackage{graphicx}
\usepackage{caption}
\usepackage{color}
\usepackage{float}
\usepackage{subfigure}
\usepackage[final]{microtype}
\usepackage{verbatim}
\usepackage{qtree}
\usepackage{tikz}
\usepackage{tikz-qtree}
\usetikzlibrary{arrows,automata,positioning}

\theoremstyle{definition} \newtheorem{Definition}{Definition}[section]
\theoremstyle{plain} \newtheorem{Theorem}[Definition]{Theorem}
\theoremstyle{plain} \newtheorem{corollary}[Definition]{Corollary}
\theoremstyle{plain} \newtheorem{lemma}[Definition]{Lemma}
\theoremstyle{definition} \newtheorem{Remark}[Definition]{Remark}
\theoremstyle{plain}
\newtheorem{proposition}[Definition]{Propostion}
\theoremstyle{plain} 
\theoremstyle{plain} 
\theoremstyle{definition}

\theoremstyle{plain}

\newcommand{\gen}[1]{\langle #1 \rangle}

\newcommand{\R}{\mathcal{R}}

\newcommand{\id}{\mbox{id}}

\newcommand{\aut}[1]{\mathop{\mathrm{Aut}}({#1})}
\newcommand{\out}[1]{\mathop{\mathrm{Out}}({#1})}

\newcommand{\T}[1]{\mathcal{#1}}
\newcommand{\CCmr}[1]{\mathfrak{C}_{#1}}
\newcommand{\CCnr}{\mathfrak{C}_{n,r}}
\newcommand{\CCn}{\mathfrak{C}_{n}}

\newcommand{\Tnr}{T_{n,r}}
\newcommand{\CTnr}{\overline{\Tnr}}
\newcommand{\Gnr}{G_{n,r}}
\newcommand{\Rnr}{\mathcal{R}_{n,r}}
\newcommand{\Bnr}{\mathcal{B}_{n,r}}
\newcommand{\TBnr}{\T{T}\Bnr}
\newcommand{\WTBnr}{\widetilde{\T{T}\Bnr}}

\newcommand{\On}{\T{O}_{n}}
\newcommand{\Onr}{\T{O}_{n,r}}

\newcommand{\TOn}{\T{T}\On}
\newcommand{\TOnr}{\T{T}\Onr}

\newcommand{\WTOnr}{\widetilde{\T{TO}_{n,r}}}

\newcommand{\WTOn}{\widetilde{\TOn}}

\newcommand{\xn}{X_{n}}
\newcommand{\xns}{\xn^{*}}
\newcommand{\xnp}{\xn^{+}}
\newcommand{\xno}{X_n^{\omega}}

\newcommand{\pxnl}[1]{X_{n}^{#1}}

\newcommand{\xnr}{X_{n,r}}
\newcommand{\xnro}{\xnr^{\omega}}
\newcommand{\xnrs}{\xnr^{*}}
\newcommand{\xnrp}{\xnr^{+}}
\newcommand{\pxnrl}[1]{X_{n,r}^{#1}}

\newcommand{\core}{\mathrm{Core}}

\newcommand{\ac}[1]{{\bf{\bar{#1}}}}

\newcommand{\lelex}{<_{\mbox{lex}}}

\newcommand{\Z}{\mathbb{Z}}
\newcommand{\N}{\mathbb{N}}
\newcommand{\sym}{\mathmbox{Sym}}

\newcommand{\dotr}{{\bf{\dot{r}}}}

\renewcommand{\restriction}{\mathord{\upharpoonright}}

\makeatletter
\renewcommand*{\eqref}[1]{%
  \hyperref[{#1}]{\textup{\tagform@{\ref*{#1}}}}%
}
\makeatother

\begin{document}
 \author{
 		Olukoya, Feyishayo\\
 		Department of Mathematics,\\
 		University of Aberdeen, \\ 
 		Fraser Noble Building,\\ 
 		Aberdeen,\\
 		\texttt{feyisayo.olukoya@abdn.ac.uk}
 	}

\title{  Automorphism towers of groups of homeomorphisms of Cantor space}
\maketitle
\begin{abstract}
We show that for any full and sufficiently transitive (i.e. \textit{flexible}) group  $G$ of homeomorphisms of Cantor space, $\aut{\aut{G}} = \aut{G}$. This class contains  many generalisations of the Higman-Thompson groups $\Gnr$,  and the Rational group $\mathcal{R}_{2}$ of Grigorchuk, Nekrashevych, and Suchanski{\u \i}. We also demonstrate that for generalisations $\Tnr$ of  R. Thompson's group $T$, $\aut{\aut{\Tnr}}= \aut{\Tnr}$. In the case of the groups $\Gnr$ and $\Tnr$ our results extend results of Brin and Guzm{\' a}n for Thompson's group $T$, and generalisations of Thompson's group $F$. 
\end{abstract}
\section{Introduction}

In this article we prove the following result:

\begin{Theorem}\label{Theorem:mainintro}
	Let $G$ be a full and flexible group of homeomorphisms of Cantor space. Then $\aut{\aut{G}} = \aut{G}$.
\end{Theorem}

We note that a group $G$ of homeomorphisms of a topological space $X$ is full if any homeomorphism $h$ of $X$ that agrees with $G$ locally is in fact an element of $G$. The flexibility condition is a transitivity condition.

The class of full and flexible groups of homeomorphisms of Cantor space contains many groups of interest: Thompson's group $V$, and the Higman-Thompson groups $G_{n,r}$ generalising $V$; the rational group $\mathcal{R}_{2}$; the R{\"o}ver group \cite{Rover_99} a simple overgroup of $V$ and the Nekrashevych groups $V_{n}(G)$  \cite{Nek_selfsimilargroups} (overgroups of $G_{n,1}$) generalising the R{\"o}ver group, the groups $V_{n}(T)$, for $T$ a \textit{partially invertible} transducer, introduced in the article \cite{DonovenOlukoya}. 

Therefore as a Corollary of Theorem~\ref{Theorem:mainintro}, we have:

\begin{corollary}
	The following groups G all have the property $\aut{\aut{G}}=\aut{G}$.
	\begin{enumerate}
		
		\item The rational group $\T{R}_2$,
		\item The Higman‐Thompson groups $\Gnr$,
		\item The Nekrashevych groups $V_{n}(G)$,
		\item The groups $V_{n}(T)$ of the article \cite{DonovenOlukoya}.
		
	\end{enumerate}
\end{corollary}

The groups $\T{R}_{n}$ of rational homeomorphisms of Cantor spaces $\CCn$ are introduced in the article  \cite{GriNekSus} of Grigorchuk, Nekreshevych and Suchanski{\u \i}. These consists of homeomorphisms which have finitely many `local actions' and so are homeomorphisms induced by finite state machines called transducers. It is a result of \cite{GriNekSus} that $\R_{n} \cong \R_{2}$ for all $n$. We note that it remains open whether or not $\aut{\T{R}_{2}} = \T{R}_{2}$.

The Thompson groups $F, T$ and $V$ were introduced by R. J. Thompson in the  1960's in connection to questions in logic \cite{ThompsonNotes}. The groups $T$ and $V$ were the first examples of finitely presented infinite simple groups. Higman \cite{Higmanfpsg} generalised $V$ to an infinite family $G_{n,r}$ of groups which are either simple ($n$ is even) or have a simple derived subgroup. The automorphisms of the groups $F$ and $T$ were analysed  in the seminal paper \cite{MBrin2}. In that paper Brin characterises the automorphisms of Thompson's group $F$ and $T$. The follow-up paper \cite{MBrinFGuzman} analyses automorphisms of generalisations of the groups $F$ and $T$ including the groups $F_n$ and $T_{n}$. In \cite{MBrinFGuzman} it is shown that the automorphism towers of the groups $F_{n}$ all have height $1$. However, the techniques used in analysing the automorphisms of these groups do not extend to analyse automorphisms of the groups $\Tnr$, when $r$ is not equal to $n-1,$ and $\Gnr$ for all valid $n$ and $r$.

The recent papers \cite{AutGnr}   and \cite{OlukoyaAutTnr} address this gap. The paper \cite{AutGnr} shows that the group of automorphisms  $\Gnr$ is a subgroup of the group of \emph{rational homeomorphisms} of the Cantor space $\CCnr$ and the paper \cite{OlukoyaAutTnr} extends this result to the group $\Tnr$. 

Although the group $\Tnr$ is a flexible group of homeomorphisms of Cantor space, it is not a full group. However, we are again able to prove that Theorem~\ref{Theorem:mainintro} holds for the group $\Tnr$:

\begin{Theorem}\label{Theorem:AutTnr}
	For the group $\Tnr$, $\aut{\aut{\Tnr}} = \aut{\Tnr}$.
\end{Theorem}

We say a few words about the proof.

 Let $G$ be full and flexible group of homeomorphisms of  Cantor  space $\T{X}$. We prove that $G$ satisfies the conditions of Rubin's Theorem \cite{Rubin} and so $\aut{G}$ is isomorphic to the normaliser of $G$ in the group of homeomorphisms of $\T{X}$. It then follows that  $\aut{\aut{G}}$ also satisfies the hypothesis of Rubin's theorem and is isomorphic to the normaliser of $\aut{G}$ in the group of homeomorphisms of $\T{X}$. At this stage the proof breaks up into two parts. Firstly, we show that an element of $\aut{G}$ which acts as the identity on a non-empty  open subset of $\T{X}$, must in fact be an element of $G$.  The second step is to observe that any full group of homeomorphisms of Cantor space  is generated by elements which act as the identity on a non-empty open subset. Since such elements are preserved by conjugation, it follows, as $\aut{\aut{G}}$ is the normaliser of $\aut{G}$ in the group of homeomorphisms of $\T{X}$, that any element of $\aut{\aut{G}}$ induces an automorphism of $G$.
 
 As noted above, the group $\Tnr$ is a flexible but not full group of homeomorphisms of Cantor space. Therefore we take a different approach to handle this case. Our approach here is similar to those  in the articles \cite{MBrin2, MBrinFGuzman} and involves a characterisation of the germs of elements of  $\aut{\Tnr}$.

 The article is organised in the following manner. In Section~\ref{Section:autautcantor} we prove Theorem~\ref{Theorem:mainintro}. In Section~\ref{Section:autautTnr} we collect the relevant results and definitions from \cite{AutGnr,OlukoyaAutTnr} and we  prove Theorem~\ref{Theorem:AutTnr}.
 
 \section*{Acknowledgements}
 The author is grateful to Prof. Nekrashevych and Dr. Bleak for helpful discussions and to comments of anonymous referees on early drafts of the manuscript. The author gratefully acknowledges the support of  Leverhulme Trust Research Project Grant RPG-2017-159.

\section{Automorphism towers of full groups}\label{Section:autautcantor}
In this section we show that the automorphism towers of full groups of homeomorphisms of Cantor space have height at most one. This class includes the  generalisations  $\Gnr$ of  Thompson's group  $V$, and the  rational group $\R_{2}$. We begin by first setting up some general notation and conventions, then we define the class of groups we are concerned with and  we prove the main result.

\subsection{General notation and definitions}

For $k \in \N$ write $\N_{k}$ for the set of natural numbers bigger than or equal to $k$.

Let $\T{X}$ be a topological space. We denote by $H(\T{X})$ the group of self-homeomorphisms of $\T{X}$. For a subgroup $G \le H(\T{X})$ we denote by $N_{H(\T{X})}(G)$ the normaliser of $G$ in $H(\T{X})$.

We write functions to the right of their arguments. In keeping with this convention, for a group $G$ and elements $g,h \in G$, we write $g^h$ for the product $h^{-1}gh$.

\begin{Definition}
Let $\T{X}$ be a topological space and $G \le H(\T{X})$. Let $h \in H(\T{X})$, then $h$ is said to \emph{locally agree with $G$} if for every point  $x \in \T{X}$ there is an open neighbourhood $U$ of  $x$ and an element $g \in G$ such that $h \restriction_{U} = g\restriction_{U}$. The group $G$ is said to be \emph{full} if every element of  $H(\T{X})$ which locally agrees with $G$ is in fact an element of $G$.
\end{Definition}

\begin{Definition}
	Let $G$ be a group acting by homeomorphisms on a compact Hausdorff space $\T{X}$. Then $G$ is called \emph{flexible} if for any pair $E_1, E_2$ of proper closed sets with non-empty interior, there is a $g \in G$ with  $(E_1)g \subseteq E_2$. 
\end{Definition}

\begin{Remark}
The groups  $\Gnr$ and $\T{R}_{2}$ are  full and flexible groups of homeomorphism of Cantor space. The R{\"o}ver group \cite{Rover_99} and the Nekrashevych groups \cite{Nek_selfsimilargroups} $V_{n}(G)$ generalising the R{\"o}ver group are full and flexible groups of homeomorphism of Cantor space. 
\end{Remark}

\begin{Definition}
	Let $G$ be a group of homeomorphisms of a topological space $\T{X}$ and let $g \in G$. The \textit{support} of $g$ is the closure of the set of points $x \in \T{X}$ such that $(x)g \ne x$.
\end{Definition}

\begin{Definition}
	Let $G$ be a group acting by homeomorphisms on a compact Hausdorff space $\T{X}$. An element  $g \in G$ is said to have \textit{small support} if there is a proper, closed subset $U$ of $\T{X}$ such that  $g \restriction_{\T{X}\backslash U}$ is the identity map on  $\T{X}\backslash U$.  
\end{Definition}

\subsection{Proofs}

Throughout this section \textbf{$\T{X}$ denotes  Cantor space}.

\begin{lemma}\label{Lemma:Tnrgeneratedbyelementsofsmallsupport}
Let $G$ be a full group of homeomorphisms of $\T{X}$.  Then $G$ is generated by its elements of small support.
\end{lemma}
\begin{proof}

We first note that for any pair $x,y \in \T{X}$, there is a proper clopen set $U \subseteq \T{X}$ containing $x,y$.

Let $g \in G$ be any non-identity element. Then, since $g$ moves a point, by the observation above, there is a non-empty clopen set $E$ of $\T{X}$ such that $E \cap (E)g = \emptyset$ and $E \cup (E)g$  is not equal to  $\T{X}$. Define $h \in H(\T{X})$ such that $h$ agrees with $g$ on $E$, with $g^{-1}$ on   $(E)g$, and agrees with the identity map on $\T{X} - (E \cup (E)g)$. Since $G$ is full, $h \in G$. Moreover,  as $gh$ acts trivially on $E$ and $h^{-1}$ acts trivially on the complement of $E \cup (E)g$, then $g = (gh)h^{-1}$ is a product of elements of small support.

\end{proof}

The author is grateful for a comment of an anonymous referee for the following lemma.

\begin{lemma}\label{Lemma:elementofTbnrwhichactsastheindentityonaclosedsubsetisinTnr}
Let $G$ be a full and flexible group of homeomorphisms of $\T{X}$. Then any element of $N_{H(\T{X})}(G)$ which is an element of small support is an element of $G$. 
\end{lemma}
\begin{proof}
	Let $g$ be an element of $N_{H(\T{X})}(G)$ which acts trivially on a non-empty clopen subset $Y$ of $\T{X}$. 
	
	Let $x$ be an arbitrary point of $\T{X}$ and set $y = (x)g$. Since $G$ is a flexible, there is an $h \in G$ such that $z = (x)h \in Y$.
	
	Since $Y$ is non-empty, there are open neighbourhoods $N_{x}, N_{y}, N_{z}$ such that $N_{z} \subseteq Y$, $h: N_{x} \to N_{z}$ is a homeomorphism, and $g^{-1}: N_{y} \to N_{x}$ is a homeomorphism.
	
	Consider the element $k = g^{-1}hg$. Since $g \in N_{H(\T{X})}(G)$,  $k \in G$, moreover, since $g \restriction_{N_{z}}$ is the identity map, $k$ coincides with $g^{-1}h$ on $N_{y}$. Therefore, $g^{-1}$ coincides with $kh^{-1}$ on $N_{y}$ and $g$ coincides with $hk^{-1}$ on $N_{x}$. Since $x \in \T{X}$ was arbitrarily chosen and $hk^{-1} \in G$, it follows that for any $x \in \T{X}$, there is a neighbourhood $N_{x}$ of $x$ such that $g$ coincides with an element of $G$ on $N_{x}$. Since $G$ is a full group, $g \in G$.
\end{proof}

\begin{lemma}\label{Lemma:endomorphismsofautfullgroups}
	Let $G$ be a full and flexible group of homeomorphisms of $\T{X}$. Let $h \in H(\T{X})$ be an element such that $h^{-1}N_{H(\T{X})}(G) h \subseteq N_{H(\T{X})}(G)$. Then $h^{-1}G h \subseteq  G$. 
\end{lemma}
\begin{proof}
	Let $g \in G$ be an element of small support, then,  $h^{-1}gh = g^{h}$ is again an element of small support and $g^{h} \in N_{H(\T{X})}(G)$. Therefore by Lemma~\ref{Lemma:elementofTbnrwhichactsastheindentityonaclosedsubsetisinTnr} $g^{h}$ in $G$. Since $G$ is generated by its elements of small support, $G^{h} = \{ k^{h} \mid k \in G \} \subseteq G$.
\end{proof}

\begin{Theorem}\label{Theorem:fullflexibleexpansion}
Let $G$ be a full and flexible group of homeomorphisms of $\T{X}$.  Suppose that $N_{H(\T{X})}(G) \cong  \aut{G}$ and $N_{{H(\T{X})}}(N_{H(\T{X})}(G)) \cong \aut{\aut{G}}$. Then $\aut{\aut{G}} = \aut{G}$.
\end{Theorem}
\begin{proof}

This is a straight-forward consequence of Lemma~\ref{Lemma:endomorphismsofautfullgroups}.
\end{proof}

We now show that any full and flexible group of homeomorphisms of $\T{X}$ satisfies the hypothesis of Theorem~\ref{Theorem:fullflexibleexpansion}. We make use of  of Rubin's Theorem \cite{Rubin}:

\begin{Theorem}\label{Thereom:RubinCantor}
	Let $\gen{X,G}$ and $\gen{Y,H}$ be space-group pairs. Assume that $X$ is Hausdorff, locally compact, with no isolated points and that for every $x \in X$ and every open neighbourhood $U$ of $x$, the set $\{xg \mid g \in G \mbox{ and } g\restriction_{(X-U)} = \id\restriction_{(X-U)} \}$ is somewhere dense. Further assume that the same holds for $\gen{Y,H}$. Then for a given group isomorphism $\phi: G \to H$, there is a homeomorphism $\varphi: X \to Y$ such that $g\phi = \varphi^{-1}g \varphi$ for every $g \in G$.
\end{Theorem}

\begin{proposition}\label{prop:fullflexiblerubin}
	Let $G$ be a full and flexible group of homeomorphisms of  $\T{X}$. Then $G$ satisfies the conditions of Rubin's theorem  
\end{proposition}
\begin{proof}

	Let $x \in \T{X}$ and $U$ be any open neighbourhood of $x$. Let $y \in U$, $y \ne x$ be arbitrary. We may find a clopen subset $U'$ of $U$ containing $x$ and $y$. Let $M_{x} \subseteq U'$ and $M_{y}\subseteq U'$ be a pair of  proper clopen neighbourhoods of $x$ and $y$ respectively. We may assume that $M_{x}$ and $M_{y}$ are disjoint since $y \ne x$. Since $G$ is flexible we may find an element $g \in G$ such that $M_{(x)g}:=(M_{x})g \subseteq M_{y}$.

	Let $h \in H(\T{X})$ be defined such that $h$ agrees with $g$ on $M_{x}$, with $g^{-1}$ on $M_{(x)g}$, and with the identity map on the complement of $M_{x} \cup M_{(x)g}$. Then $h \in G$ since $G$ is full. Moreover, the support of $h$ is contained entirely in $U$.
	
	Therefore for any element $y \in U$, and any neighbourhood $V$ of $y$, there is an element  $h \in G$ such that $(x)h \in V$ and $h\restriction_{(\T{X}-U)} = \id \restriction_{(\T{X}-U)}$. In particular, $$\{xg \mid g \in G \mbox{ and } g\restriction_{(X-U)} = \id\restriction_{(X-U)} \}$$ is dense in $U$.
	
\end{proof}

\begin{lemma}\label{Lemma:normaliserisautomorphism}
	Let $G$ be a full and flexible group of homeomorphisms of  $\T{X}$. Then the centraliser of $G$ in $H(\T{X})$ is trivial. 
\end{lemma}
\begin{proof}
	This follows since by Proposition~\ref{prop:fullflexiblerubin}, $G$ satisfies the conditions of Theorem~\ref{Thereom:RubinCantor}. For let $h \in H(\T{X})$ be any non-trivial element. Let $x \in \T{X}$ be such that $(x)h \ne x$. We may find  a clopen neighbourhood $N_{x}$ of $x$ such that $(N_{x})h$ is disjoint from $N_{x}$. Let $z \in (N_{x})h$ be a point distinct from $(x)h$.   Let $U, V \subseteq (N_{x})h$ be disjoint clopen neighbourhoods of $(x)h$ and $z$ respectively. Then there is an element $g \in G$ such that $g\restriction_{(X-(N_{x})h)} = \id \restriction_{(X-(N_{x})h)}$ and $(x)hg \in V$. Therefore $h^{g} \ne h$ as $(x)h^{g} \in V$ and $(x)h \in U$.
\end{proof}

\begin{corollary}
	Let $G$ be a full and flexible group  of homeomorphisms of $\T{X}$. Then $\aut{\aut{G}} = \aut{G}$. 
\end{corollary}
\begin{proof}
	By Propositon~\ref{prop:fullflexiblerubin} and Lemma~\ref{Lemma:normaliserisautomorphism}, $\aut{G} \cong N_{H(\T{X})}(G)$. In particular, $G \le N_{H(\T{X})}(G)$ and so $N_{H(\T{X})}(G)$ satisfies the conditions of Rubin's Theorem as well. 
	
	Therefore, $N_{H(\T{X})}(N_{H(\T{X})}(G)) = \aut{\aut{G}}$. The result now follows by Theorem~\ref{Theorem:fullflexibleexpansion}. 
\end{proof}

\begin{corollary}
The following groups G all have the property $\aut{\aut{G}}=\aut{G}$.
\begin{enumerate}
	\item The Higman‐Thompson groups $\Gnr$.
	\item The rational group $\T{R}_2$.
	\item The Nekrashevych groups $V_{n}(G)$.
	\item The groups $V_{n}(T)$ introduced in the paper \cite{DonovenOlukoya}.
	
\end{enumerate}
	
\end{corollary}

We note that in general Lemma~\ref{Lemma:Tnrgeneratedbyelementsofsmallsupport} fails for an arbitrary compact Hausdorff space. Jim Belk in a personal communication gives an example of a full and flexible group of homeomorphisms of the circle that is not generated by its elements of small support. Therefore, the strategy we employ to prove Theorem~\ref{Theorem:fullflexibleexpansion} does not go through for an arbitrary compact Hausdorff space. 

We observe that by Lemma~\ref{Lemma:endomorphismsofautfullgroups} we may say something about homeomorphisms of Cantor space  that normalise $\aut{\Gnr}$. Specifically, such homeomorphisms must also normalise $\Gnr$. Therefore, they must be rational and one-way synchronzing by results of \cite{OlukoyaAutTnr, AutGnr}.

In particular we have,

\begin{corollary}
	Let and $h \in H(\CCnr)$. Suppose that $h^{-1} \aut{\Gnr} h \subseteq \Bnr $, then $h$ is rational and can be induced by a synchronizing transducer. 
\end{corollary}

The following question about $\T{R}_{2}$ remains  open:

\begin{enumerate}[label={\bf Q.\arabic*}]
	\item Is the equality  $\aut{\R_{2}} = \R_{2}$ valid?
\end{enumerate}


\section{Automorphism tower of \texorpdfstring{$\aut{\Tnr}$}{Lg} }\label{Section:autautTnr}
The proof strategy in the previous section does not go through for groups which are flexible but not full. Therefore, we need to take a slightly different approach for such groups. In this section we get around this issue for the group $\Tnr$ and show that $\aut{\aut{\Tnr}} = \aut{\Tnr}$. Our approach for $\Tnr$ begins in much the same way as the previous section: the group $\Tnr$ is generated by its elements of small support and is a Rubin group on the circle. The divergence occurs as the proof of   Lemma~\ref{Lemma:elementofTbnrwhichactsastheindentityonaclosedsubsetisinTnr} as given in the previous section, does not go through verbatim in this context. To recover Lemma~\ref{Lemma:elementofTbnrwhichactsastheindentityonaclosedsubsetisinTnr}, we make use of the group of germs of elements of $\aut{\Tnr}$. This was a key tool in the articles \cite{MBrin2, MBrinFGuzman}. We characterise the group of germs of elements of $\aut{\Tnr}$ and from our characterisation, Lemma~\ref{Lemma:elementofTbnrwhichactsastheindentityonaclosedsubsetisinTnr} naturally arises. It is possible that a similar approach works for other flexible but not full groups of homeomorphisms. 

We require slightly technical machinery to characterise the germs of elements of $\aut{\Tnr}$. Specifically, we assume familiarity with the characterisation of $\aut{\Tnr}$, as  given in \cite{OlukoyaAutTnr}, as a group of homeomorphisms of Cantor space induced by transducers. For background reading on such groups one should also consult the articles \cite{GriNekSus,AutGnr}. 

\subsection{Words and Cantor space}
Set $\xn:= \{0,1,\ldots, n-1\}$, and $\dotr:= \{\dot{1}, \dot{2}, \ldots, \dot{r}\}$. Set $\xns$ to be the set of all finite words (including the empty word) in the alphabet $\xn$, and set $\xnrs:= \dotr \times \xns$. We identify $\xnrs$ with the set consisting of the empty word and  all finite words over the alphabet $\dotr \sqcup \xn$ which begin with an element of $\dotr$ and contain no other letters from $\dotr$. We shall use $\epsilon$ for the empty word. Set $\xnp := \xns\backslash \{\epsilon\}$ likewise set $\xnrp := \xnrs\backslash \{\epsilon\}$. For $j \in \N$ we denote by $\pxnl{j}$ and $\pxnrl{j}$ the set of all words in $\xns$ and $\xnrs$ of length $j$. 

Set $\xno$ to be the set of all infinite words over the alphabet $\xn$ and set $\xnro := \dotr \times \xno$. We identify $\xnro$ with the set of all infinite words over the alphabet $\dotr \sqcup \xn$ which begin with a letter in $\dotr$ and have no other occurrence of an element of $\dotr$. 

The sets $\xno$ and $\xnro$ are homeomorphic to Cantor space with the usual topology. We denote by $\CCn$ the space $\xno$ and $\CCnr$ the space $\xnro$.

Given a word $\nu \in \xnrp \sqcup \xns$ we set $U_{\nu}:= \{\nu \rho \mid \rho \in \CCn\}$, if $\nu = \epsilon$. Depending on the context $U_{\epsilon}$ represents either $\CCnr$ or $\CCn$, whenever we use this notation, it will be clear which set is meant. 

The set $\xns$ is a monoid under concatenation. We also observe that concatenating an element of $\xnrp$ with an element of $\xns$ results in an element of $\xnrs$. We represent this operation by simply writing the elements beside each other. 

We may partially order the elements of $\xns$ and $\xnrs$ as follows. Let $X$ be either $\xnrs$ or $\xns$. Given two elements $\nu, \eta \in X$  we write $\nu \le \eta$ if $\nu$ is a prefix of $\eta$. If $\nu \not\le \eta$ and $\eta \not\le \nu$, then we say $\nu$ is incomparable to $\eta$ and write $\nu \perp \eta$ to denote this. Let $\nu, \eta \in X$ such that $\nu \le \eta$, then we write $\eta - \nu$ for the word $\tau \in \xns \sqcup \xnrs$ such that $\eta = \nu \tau$. 

\begin{Definition}
	Let $X^{\ast}$ be one of $\xnrs$ or $\xns$. A subset $\ac{u}$ of $X^{\ast}$ is called an \emph{antichain (for $X^{\ast}$)} if $\ac{u}$ consists of pairwise incomparable elements. An antichain $\ac{u}$ for $X^{\ast}$ is called \emph{complete} if for any word $\nu \in X^{\ast}$ either there is some element of $\ac{u}$ which is a prefix of $\nu$ or $\nu$ is a prefix of some element of $\ac{u}$. 
\end{Definition}

The natural ordering on  the sets  $\dotr$ and $\xn$ induced from $\N$, means that we may consider the lexicographic ordering $\lelex$ on the sets $\xnrs$ and $\xns$. That is for $\nu, \mu \in \xns$ or $\nu, \mu \in \xnrs$, $\nu \lelex \mu$ if either $\nu$ is a prefix of $\mu$ or there are words $u \in \xnrs \sqcup
\xns$, $v,w \in \xns$ and $a, b \in \dotr$ or $a,b \in \xn$ such that  $a$ is less than $b$ in the natural ordering on $\dotr$ or $\xn$ and $\nu = uaw$ and $v=ubw$.

In this article we {\bfseries{assume that all antichains are ordered lexicographically}}.

\subsection{The subgroups \texorpdfstring{$\Bnr$ and $\TBnr$ \mbox{ of } $\T{R}_{n,r}$ }{Lg}}
The article \cite{AutGnr} shows that the subgroup  $\Rnr$  of $H(\CCnr)$ is the image under a topological conjugacy of the group $\mathcal{R}_{2} \le H(\CCmr{2,1})$. We specify some subgroups of $\Rnr$ based on a combinatorial property of the transducer inducing the homeomorphisms.

We begin with a combinatorial property of the transducers.

\begin{Definition}
	A transducer (initial or non-initial) $T = \gen{X_I, X_O, Q_T, \pi_T, \lambda_T}$ is said to be \emph{synchronizing at level $k$} for some natural number $k \in \N$, if there is a map $\mathfrak{s}: X_I^{k} \to  Q_T$ such that for a word $\Gamma \in X_I^{{k}}$ and for any state $q \in Q_T$ we have $\pi_{T}(\Gamma, q) = (\Gamma)\mathfrak{s}$. We say that $T$ is \emph{synchronizing} if it is synchronizing at level $k$ for some $k \in \N$.
\end{Definition}

We will denote by $\core(T)$ the sub-transducer of $T$ induced by the states in the image of $\mathfrak{s}$. We call this sub-transducer the \emph{core of $T$}. If $T$ is equal to its core then we say that $T$ is \emph{core}. Viewed as a graph $\core(T)$ is a strongly connected transducer. 

\begin{Definition}
	If $T$ is an initial transducer $T_{q_0}$  which is invertible, then we say that $T_{q_0}$ is bi-synchronizing if both $T_{q_0}$ and its inverse are synchronizing. Note that when $T$ is synchronous, then we shall say $T$ is bi-synchronizing if $T$ and its inverse are synchronizing.
\end{Definition}

We say that a transducer $A_{q_0}$ over $\CCnr$ is synchronizing at level $k$ for a natural number $k \in \N$ if given any word $\Gamma$ of length $k$ in $\xns$ the active state of $A_{q_0}$ when $\Gamma$ is processed from any non-initial state of $A_{q_0}$ is completely determined by $\Gamma$. We say that $A_{q_0}$ is synchronizing if it is synchronizing at level $k$ for some $k \in \N$. Thus we may also extend the notions of `core'  for synchronizing transducers over $\CCnr$. If the inverse of  $A_{q_0}$ is also synchronizing, then we say that $A_{q_0}$ is bi-synchronizing.

\begin{Definition}
	Let $T_{q_0}$ be an initial synchronizing transducer for $\CCn$ or $\CCnr$, then $T_{q_0}$ is said to have trivial core if $\core(T_{q_0})$ consists of the single state transducer inducing the identity homeomorphism on $\CCn$.
\end{Definition}

The set $\Bnr$ of all homeomorphisms in $\Rnr$ which may be represented by a bi-synchronizing transducers forms a subgroup of $\Rnr$ (\cite{AutGnr}). Let $\TBnr$ be the subgroup of $\Bnr$ of elements which either preserve or reverse the standard circular ordering of $\CCnr$ and preserves the set $\{ \nu \xi^{\omega} \mid \nu \in \xnrp, \xi \in \{0, n-1 \}\}$.  The subgroup of $\Bnr$ consisting of all elements which can be represented by a bi-synchronizing transducer with trivial core is the Higman-Thompson group $\Gnr$. The condition that the core is trivial means that elements of $\Gnr$ are homeomorphisms of $\CCnr$ given by prefix replacement maps. That is, given $g \in \Gnr$, there are complete antichains  $\ac{u} = \{u_0, \ldots, u_{l} \}$, $\ac{v}= \{v_0, \ldots, v_l\}$ for $\xnrs$ and a map $\tau \in \sym(\{0,1,\ldots,l\})$ such that for $0 \le a \le l$ and $\rho \in \CCn$, $(u_a\rho)g = v_{(a)\tau}\rho$. Let  $\CTnr$ be the subgroup of $\Gnr$ consisting of those elements $g$ such that there are complete antichians  $\ac{u} = \{u_0, \ldots, u_{l} \}$, $\ac{v}= \{v_0, \ldots, v_l\}$ for $\xnrs$ and $ b \in \{0,1,\ldots,l\}$ such that for $0 \le a \le l$ and $\rho \in \CCn$, $(u_a\rho)g = v_{(a+b)\mod{l}}\ \rho$. We observe that $\CTnr$ is isomorphic to the Higman-Thompson group $\Tnr$. 

The following result is proved in \cite{AutGnr}.

\begin{Theorem}[Bleak, Cameron, Maissel, Navas and O]
	$\aut{\Gnr} \cong \Bnr$.
\end{Theorem}

\subsection{Actions on Cantor space and the circle}
Let $S_{r}$ be the circle corresponding to the interval $[0,r]$ with the end points identified. For $n \in \N_{2}$, we set $\Z[1/n]:= \{ \sfrac{a}{n^c} \mid a, c \in \Z\}$ the set of $n$-adic rationals.  Let $\T{N} \le H(S_{r})$ be the subgroup of orientation preserving elements which induce bijections from the $\Z[1/n] \cap [0,r)$ to itself. 

Let $\simeq$ be the equivalence relation on $\CCnr$ given by $\rho \simeq \delta$ if either there is a word $\nu \in \xnrs$ and $a \in \xn \backslash \{0\}$ or $ a \in \dotr \backslash \{\dot{0}\}$  such that $\rho= \nu a000\ldots$ and $\delta = \nu a-1n-1n-1n-1\ldots$ or $\rho = \dot{0}00\ldots$ and $\delta= \dot{{(r-1)}}n-1n-1\ldots$. 

Let $g \in \T{N}$, and let $x \in [0,r]$. We observe that $x$ has non-unique  $n$-ary expansion in $\CCnr$ precisely when $x \in \Z[1/n] \cap  (0,r)$. In this case the $n$-ary expansions of $x$ take the form $\mu a000\ldots$ and $\delta = \mu a-1n-1n-1n-1\ldots$ for some $\mu \in \xnrs$ and $a \in \xn \backslash \{0\}$ or $ a \in \dot{r} \backslash \dot{0}$.  Let $\bar{x}_{1}, \bar{x}_{2} \in \CCnr$ be the $n$-ary expansions of $x$, such that $\bar{x}_{1} \lelex \bar{x}_{2}$ in the lexicographic ordering of $\CCnr$ induced by the natural order on the sets $\dotr$ and $\xn$ (set $\bar{x}_{1} = \bar{x}_{2}$ if $x$ has a unique $n$-ary expansion).  Observe that for $\nu \in \xnrs$ the clopen set $U_{\nu}$ corresponds to an  interval $[\sfrac{a}{n^c}, \sfrac{(a+1)}{n^c}] \subseteq [0,r]$ for some $a \in \N$ and $c \in \Z$. 
Let $\bar{g} : \CCnr \to \CCnr$ be defined by for $x \in [0,r]$, and $a = 1,2$, $(\bar{x}_{a})\bar{g} = \overline{(x)g}_{a}$. As elements  of $\T{N}$ map $\Z[1/n]\cap [0,r]$ bijectively into itself, the map $\bar{g}$ is well-defined. Since every element of $\CCnr$ corresponds to $\bar{x}_{a}$, $a = 1,2$, for some $x \in [0,r]$, and since $g$ is a homeomorphism of $S_{r}$, one may deduce that $\bar{g}$ is a homeomorphism of $\CCnr$. More specifically the map $\iota: \T{N} \to H(\CCnr)$ by  $g \mapsto \bar{g}$ is an injective homomorphism.

We write $\T{N}$ again for the image of $\T{N}$ under $\iota$. We note that $\TBnr$ is a subgroup of $(\T{N})\iota$ and we likewise we do not distinguish between $\TBnr$ and $(\TBnr)\iota^{-1}$. Therefore, we dually veiw $\TBnr$ as acting on Cantor space and on the circle and move freely between these two points of views.

The following result is from \cite{OlukoyaAutTnr}:

\begin{Theorem}[O]
	$\aut{\Tnr} \cong \TBnr$.
\end{Theorem}

We write $\WTBnr$ for the subgroup of $\TBnr$ of elements which induce orientation preserving elements of $H(S_r)$.

We set $\TOnr$ for the set of cores of $\TBnr$ and $\WTOnr$ for the set of cores of elements of $\WTBnr$. There is a product on $\TOnr$ which makes it into a group with $\WTOnr$ an index two subgroup (see \cite{OlukoyaAutTnr, AutGnr}).

The following result is from \cite{OlukoyaAutTnr}:

\begin{Theorem}
	The group $\TOnr$ is isomorphic to $\out{\Tnr}$.
\end{Theorem}

\subsection{Automorphism tower of \texorpdfstring{$\Tnr$}{Lg}} \label{Section:AutomorphismtowerofTnr}

\begin{Definition}
	Let $G \le H(S_{r})$ and $D \subseteq S_{r}$. Then $G$  is said to act o-$k$-transitively if for every pair $x_1, x_2, \ldots, x_k$ and $y_1,y_2, \ldots, y_{k}$ of $k$-tuples of points in $D$ such that $x_1 < x_2 < \ldots < x_k$ and $y_1 < y_2 < \ldots < y_{k}$ on some closed interval of $S_{r}$ (order induced from the ordering on $\mathbb{R}$), then there is a $g \in G$ such that $(x_i)g = y_i$ for all $1 \le i \le k$.
\end{Definition}

\begin{Remark}
	The group $\Tnr$ acts o-$k$-transitively on the set $\Z[1/n] \cap  [0,r)$. Thus since $\Tnr \le \TBnr$, $\TBnr$ acts o-$k$-transitively on $\Z[1/n] \cap  [0,r)$. 
\end{Remark}

The following result is due to McCleary  and Rubin (\cite{SMcClearyMRubin}) and is an analogue of Theorem~\ref{Thereom:RubinCantor} for the circle.

\begin{Theorem}[Mcleary and Rubin]\label{Theorem:RubinCircle}
	Let $G$ be a group  acting on the circle $S_{r}$ by orientation preserving homeomorphisms. Assume that $G$ acts o-$3$-transitively on a dense subset of $S_{r}$. Then for each automorphism $\alpha$ of $G$, there is a unique element $h$ of $H(S_{r})$ such that $(f)\alpha = h^{-1}f h$ for every $f \in G$. 
\end{Theorem}

As a corollary we have:

\begin{corollary}
	$\aut{\TBnr} \cong N_{H(S_{r})}(\TBnr)$.
\end{corollary}

We show that for $h \in H(S_{r})$ such that $h^{-1}\TBnr h \subseteq \TBnr$, $h^{-1} \Tnr h \subseteq \Tnr$. We do this by characterising the germs of elements of $\TBnr$ at a fixed point.

We begin with the following lemma.

\begin{lemma}\label{Lemma:coincideonopenimpliescoreequal}
	Let $A_{q_0}, B_{p_0} \in \TBnr$ and suppose that there is an open subset $U$ of $\xn$ such that $h_{q_0} \restriction_{U} = h_{p_0}\restriction_{U}$. Then $\core(A_{q_0}) = \core(B_{p_0})$.
\end{lemma} 
\begin{proof}
	This follows from the fact that the core is strongly connected and $A_{q_0}$ and $B_{p_0}$ are minimal.
\end{proof}

For $x \in S_{r}$ write $\WTOnr^{x}$ for the core of those elements of  $\WTBnr$ which fix the element $x$. We note that $\WTOnr^{x}$ is a group. If $x$ is of the form $\eta w^{\omega}$ for some $w \in \xnp$, then $\WTOnr^{x}$ is precisely the subgroup of $\WTOnr$ consisting of those elements $T \in \WTOnr$ for which the unique state $q \in Q_{T}$ satisfying $\pi_{T}(w,q) = q$ has $\lambda_{T}(w,q) = q$. We note, consequently that for $x \in S_{r} \cap \Z[1/n]$, $\WTOnr^{x} = \WTOnr$.

\begin{lemma}\label{Lemma:germs}
	Let $x \in S_{r}$ and write $\mathscr{TB}_{x,x}$ for the group of germs at $x$ of the elements of $\WTBnr$ and $\mathscr{T}_{x,x}$ for the group of germs at $x$ of elements of $\Tnr$. Then
	\begin{enumerate}
		\item For $x \in [0,r] \cap \Z[1/n]$, $\mathscr{TB}_{x,x} \cong \WTOnr \times \Z \times \Z$ and $1 \times \Z \times \Z \cong \mathscr{T}_{x,x}$;
		\item For $x \in [0,r] \cap (\mathbb{Q} \backslash \Z[1/n])$ $\mathscr{TB}_{x,x} \cong \WTOnr^{x} \times \Z$ and $1 \times \Z \cong \mathscr{T}_{x,x}$;
		\item For $x \in [0,r]$ is irrational  $\mathscr{TB}_{x,x} \cong \WTOnr^{x} $ and $1 \cong \mathscr{T}_{x,x}$.
	\end{enumerate}
	      
\end{lemma}
\begin{proof}
	Suppose that $a \in S_{r} \cap \Z[1/n]$. Then there are $\tau \lelex \tau' \in \xnrp$ such that $\tau(n-1)^{\omega}, \tau' 0^{\omega}$ are the $n$-adic expansions of $a$. We note that there is an $j \in \N$ such that $U_{\tau' 0^{j}}$ corresponds to a subset $[a, b]$ of $[a, a+ \delta)$ and $U_{\tau (n-1)^{j}}$ corresponds to a subset $[b', a]$ of $(a- \delta, a]$.
	
	Let $f,g \in \TBnr$ both fix the point $a$ and let $\delta > 0$ be such that $f$ agrees with $g$ on $(a-\delta, a+ \delta)$ so that $f,g$ represent the same element of $\mathscr{TB}_{x,x}$.
	
	 Let $A_{q_0}$ and $B_{p_0}$ be the transducers corresponding to the elements $f,g \in \TBnr$. Since $f,g$ agree on $[a, a+ \delta)$. There is a neighbourhood of $\nu'$ on which $A_{q_0}$ and $B_{p_0}$ agree. In particular it follows that $A_{q_0}$ and $B_{p_0}$ have the same core.  Let $T = \core(A_{q_0}) = \core(B_{p_0}) \in \WTOnr$.

	Now as $A_{q_0}$ and $B_{p_0}$ are both synchronizing there is a $k' \in \N_{j}$ such that $\pi_{A}(\tau'0^{k'}, q_0) = \pi_{B}(\tau' 0^{k'} p_0) = p_{l(0)}$, the unique state of $T$ with $\pi_{T}(0, p_{l(0)}) = p_{l(0)}$. Since $A_{q_0}$ and $B_{p_0}$ agree on $U_{\tau' 0^{j}}$ we have, $\lambda_{A}(\tau'0^{k'}, q_0) = \lambda_{B}(\tau' 0^{k'}, p_0)$.
	
	In a similar way we conclude that there is a $k \in \N_{j}$ such that $\pi_{A}(\tau(n-1)^{k}, q_0) = \pi_{B}(\tau (n-1)^{k} p_0)=p_{l(n-1)}$, the unique state of $T$ with $\pi_{T}(n-1, p_{l(n-1)}) = p_{l(n-1)}$ and $\lambda_{A}(\tau(n-1)^{k}, q_0) = \lambda_{B}(\tau' (n-1)^{k}, p_0)$.
	
	Given two elements $A_{q_0}, B_{p_0} \in \WTBnr$ such that $\core(A_{q_0}) = \core(B_{p_0}) = T \in \WTOn$ and there are $k, k' \in \N$ so that 
	
	\begin{itemize}
	\item $\pi_{A}(\tau'0^{k'}, q_0) = \pi_{B}(\tau' 0^{k'} p_0) = p_{l(0)}$, the unique state of $T$ with $\pi_{T}(0, p_{l(0)}) = p_{l(0)}$, and  $\tau' \le \lambda_{A}(\tau'0^{k'}, q_0) = \lambda_{B}(\tau' 0^{k'}, p_0)$;
	\item  $\pi_{A}(\tau(n-1)^{k}, q_0) = \pi_{B}(\tau (n-1)^{k}, p_0)$, the unique state of $T$ with $\pi_{T}(n-1, p_{l(n-1)}) = p_{l(n-1)}$, and $\tau \le \lambda_{A}(\tau(n-1)^{k}, q_0) = \lambda_{B}(\tau' (n-1)^{k}, p_0)$.	
   \end{itemize}
   Then $A_{q_0}, B_{p_0}$ coincide on the neighbourhoods $U_{\tau(n-1)^{k}}$ and $U_{\tau'0^{k'}}$. In particular the maps $h_{q_0}, h_{p_0}$, as maps of $S_{r}$ represent the same element of $\mathscr{TB}_{a,a}$.
   
   Let $A_{q_0}$ and $B_{p_0}$ represent the same element of $\mathscr{TB}_{a,a}$ and let $T = \core(A_{q_0}) = \core(B_{p_0})$. Let $i,j \in \N$ be minimal such that $\pi_{A}(\tau'0^{i}, q_{0}) = p_{l(0)}$ and $\pi_{B}(\tau'0^{j}, p_{0}) = p_{l(0)}$. Notice that since $\lambda_{T}(0, p_{l(0)}) =  0$, then $e_{A_{q_0}}:=\lambda_{A}(\tau'0^{i}, q_{0}) - \tau'0^{i} = \lambda_{B}(\tau'0^{j}, p_{0}) - \tau'0^{j}$. In a similar way for $l, m\in \N$ be minimal such that $\pi_{A}(\tau (n-1)^{l}, q_{0}) = p_{l(n-1}$ and $\pi_{B}(\tau (n-1)^{m}, p_{0}) = p_{l(n-1)}$,  $\lambda_{T}(n-1, p_{l(n-1)}) = n-1$, then $d_{A_{q_0}}:=\lambda_{A}(\tau (n-1)^{i}, q_{0}) - \tau (n-1)^{l} = \lambda_{B}(\tau(n-1)^{m}, p_{0}) - \tau(n-1)^{m}$.    
   
   Represent the elements of $\mathscr{TB}_{x,x}$ corresponding to $A_{p_0}$ by the pair $(T, d_{A_{q_0}}, e_{A_{q_0}})$. It is not hard to verify that for $D_{t_0} \in \TBnr$, with $\core(D_{t_0}) = V$, the element $E_{s_0}$ corresponding to the product $(AD)_{(p_0, t_0)}$ has core $TV$, $d_{E_{s_0}} = d_{A_{q_0}} + d_{D_{t_0}}$ and $e_{ E_{s_0}} = e_{A_{q_0}} + e_{D_{t_0}}$.
   
   Now given an element $A_{p_0} \in \WTBnr$ such that the germ at $a$ of $A_{p_0}$ is $(T, d_{A_{p_0}}, e_{A_{p_0}})$, for any pair $(d', e') \in \Z \times \Z$, we may find an element of $g \in \Tnr$ such that the germ of $A_{p_0}g$ at $a$ is $(T, d', e')$. Thus we see that the map  from $\mathscr{TB}_{a,a}$  to $\WTOnr \times \Z \times \Z$ sending the class of an element $A_{p_0} $ to the tuple $(\core(A_{p_0}), d_{A_{p_0}}, e_{A_{p_0}} )$ is an isomorphism. Moreover, the subgroup $1 \times \Z \times \Z$ corresponds to germs at $a$ of the elements of $\Tnr$ fixing the point $0 \in S_{r}$.
   
   Now suppose that $a$ is a rational element of $S_{r} \backslash \Z[1/n]$. In this case there are words $\tau \in \xnrp$ and $w \in \xns$ such that $a = \tau w^{\omega}$. For an element $T \in \WTOn$ write $p_{l(w)}$ for the unique state of $T$ such that $\pi_{T}(w, p_{l(w)}) = w$.
   
   By similar arguments to above, in this case we see that two elements $A_{p_0}$ and $B_{q_0}$ correspond to the same elements of $\mathscr{TB}_{a,a}$ precisely when $\core(B_{q_0}) = \core(A_{p_0}) = T$ and for $k \in \N$ such that  $\pi_{A}(\tau w^{k}, p_0) = \pi_{B}(\tau w^{k}, q_0) = p_{l(w)}$, we have, $\lambda_{A}(\tau w^{k}, p_0) = \lambda_{B}(\tau w^{k}, q_0)$. 
   
   Let $A_{p_0} \in \WTBnr$, let $T = \core(A_{p_0})$ in $\TBnr$ and let $k \in \N$ be such that $\pi_{A}(\tau w^{k}, p_0) = p_{l(w)}$. Set $d_{A_{p_0}}$ to be the value $\tau w^{k} - \lambda_{A}(\tau w^{k}, p_0)$. We note that since $\lambda_{T}(w, p_{l(w)}) = w$, then for any element $B_{q_0}$ representing the same element as $A$ in $\mathscr{TB}_{a,a}$, $d_{B_{q_0}} = d_{A_{p_0}}$. The map which sends the element of $\mathscr{TB}_{a,a}$ represented by $A_{p_0}$ to the element $(\core(A_{p_0}), d_{A_{p_0}})$ is an isomorphism from $\mathscr{TB}_{a,a}$ to $\WTOnr^{a} \times \Z$. Moreover the subgroup $1 \times \Z$ corresponds to germs at $a$ of the elements of $\Tnr$ fixing the point $a$.
   
   Lastly suppose $a \in S_{r}$ is irrational and let $\eta \in \xnro$ be the element corresponding to $a$. 
   
   Similar arguments show that two elements  $A_{p_0}$ and $B_{p_0}$ have the same germ at $a$ precisely when $\core(A_{p_0}) = \core(B_{q_0}) = T$. For, as $a$ is irrational, $(a)A_{p_0} = (a)B_{q_0}$, and $A_{p_0}$ and $B_{q_0}$ are minimal, it must be the case  that for a long enough prefix $\eta_1$ of $\eta$ such that $\pi_{A}(\eta_1, p_0) = \pi_{B}(\eta_1, q_0) \in Q_{T}$, $\lambda_{A}(\eta, p_0) = \lambda_{B}(\eta, q_0)$. In this case, all elements of $\Tnr$ fixing the point $a$ belong to the germ of the identity element.
\end{proof}

In the corollary below we show that any element $h \in H(S_{r})$ satisfying $h^{-1}\TBnr h \subseteq \TBnr$ must induce a map from the subgroup $\mathscr{T}_{x,x}$ of $\mathscr{TB}_{x,x}$ to the subgroup $\mathscr{T}_{(x)h,(x)h}$ of $\mathscr{TB}_{(x)h,(x)h}$. From this it  follows, since $\Tnr$ is generated by its elements of small support, that $h^{-1}\Tnr h \subseteq \Tnr$.

\begin{corollary}\label{cor:normaliserpresevesTnr}
	Let $h \in H(S_r)$ be such that $h^{-1}(\TBnr)h \subseteq \TBnr$. Then, $h^{-1} \Tnr h \subseteq \Tnr$. 
\end{corollary}
\begin{proof}
 
 Let  $\nu \in \xn^{p}$ such that $\nu \ne \xi ^{i}$ for some $\xi \in \{0, n-1\}$ and $i \in \N$. Let $a \in \dotr$ and consider the word $a \nu^{\omega}$. Note that the word $a \nu^\omega$ corresponds to a rational element $x$ of $S_{r}$. 
 
 We begin with the following observation. 
 
 By choice of $\nu$, for any pair $i,j \in \Z$ and any $\tau, \tau' \in \xnrp$ with $\tau 0^{\omega} \simeq \tau'(n-1)^{\omega}$, we may find an element $f \in  \Tnr$ which fixes $x$ and which has the following properties. In the action of $f$ on $\CCnr$  there are $k,l,m, \in \N$, $k, l \ge  \max\{|i|, |j|\}$ such that the following holds:
 \begin{itemize}
 	\item $U_{\tau0^{k}}f = U_{\tau0^{k+i}}$ and $U_{\tau'(n-1)^{l}}f = U_{\tau'(n-1)^{l+j}}$;
 	\item $f\restriction_{U_{a\nu^{m}}} =\id \restriction_{U_{a\nu^{m}}}$
 \end{itemize}  
 
  Let $A_{q_0} \in \TBnr$ be an element which fixes $x$. Then the germ of $A_{q_0}$ at $x$ is equal to the germ of $A_{q_0}f$ at $x$. In particular the germ of $A_{q_0}^{h}$ at $(x)h$ and $(A_{q_0}f)^{h}$ at $(x)h$ coincide. From this it follows that $A_{q_0}^{h}$ and $A_{q_0}^{h}f^{h}$ have the same core. Since $\core(A_{q_0}^{h}f^{h}) = \core(A_{q_0}^{h})\core(f^{h})$, and $\TOnr$ is a group, this is true precisely when $f^{h}$ has trivial core. Therefore we see that $f^{h} \in \Tnr$.

 Let $g \in \Tnr$ be any element that fixes a point $y \in S_{r} \cap \Z[1/n]$. Let $\tau, \tau' \in \xnrp$ be such that $\tau 0^{\omega} = \tau (n-1)^{\omega}$ are the distinct $n$-adic expansions of $y$. By Lemma~\ref{Lemma:germs}, the germ of $g$ at $y$ corresponds to a pair $(i,j) \in \Z \times \Z$. By the preceding paragraph we may find an element $f \in \Tnr$, such that $h^{-1}f h \in \Tnr$ and the germ of $f$ at $y$ coincides with the germ of $g$ at $y$. This means that the germs of $g^{h}$ and $f^{h}$ at $(y)h$ coincide. Therefore $g^{h}$ must again be an element of $\Tnr$, since it has trivial core.
 
 We next observe that it is a standard result that $\Tnr$ is generated by its elements of small support. Moreover, any element of $\Tnr$ of small support must fix some $n$-adic rational point. Thus, for any element $g \in \Tnr$ of small support $h^{-1}gh \in \Tnr$. From this we deduce that $h^{-1}\Tnr h \subseteq \Tnr$ as required.
\end{proof}

\begin{Theorem}
	$\aut{\aut{\Tnr}} = \aut{\Tnr}$.
\end{Theorem}
\begin{proof}
	This follows from Corollary~\ref{cor:normaliserpresevesTnr}.
\end{proof}

\bibliographystyle{amsplain}
\bibliography{references}

\end{document}